\documentclass[10pt,twocolumn,twoside]{IEEEtran}

\IEEEoverridecommandlockouts
\usepackage{graphics} % for pdf, bitmapped graphics files
\usepackage{epsfig} % for postscript graphics files
\usepackage{times} % assumes new font selection scheme
\usepackage{amsmath} % assumes amsmath package installed
\usepackage{breqn}
\usepackage{xfrac}
\usepackage[numbers,sort&compress]{natbib}
\usepackage{epstopdf}
\usepackage{xcolor}
\usepackage{amssymb}  % assumes amsmath package installed
\usepackage{amsthm}
\usepackage{amsmath}
\usepackage{upgreek}
\usepackage{epstopdf}
\usepackage{subfigure}
\usepackage{mathtools}
\usepackage{multirow}  % for `\multirow` macro
\usepackage{lipsum} %%% to write big equation in a single column
\usepackage{cuted}
\usepackage{epstopdf}
\usepackage{array,amssymb,booktabs}
\usepackage{enumitem}
\newcommand{\tallstrut}{\vphantom{\frac{5_A}{4,10^3}}} % a typographic strut
\usepackage{graphicx}
\graphicspath{{figures/}}
\DeclareMathOperator{\sig}{sig}

\newtheorem{thm}{Theorem}

\newtheorem{assumption}{Assumption}

\newtheorem{remark}{Remark}

\newtheorem{lemma}{Lemma}

\title{A Partially Distributed Fixed-Time Economic Dispatch Algorithm with Kron's Modeled Power Transmission Losses}
\author{Shivanshu Tripathi, Anoop Jain, and Abhisek K. Behera% <-this % stops a space
\thanks{Shivanshu Tripathi and Anoop Jain are with the Department of Electrical Engineering, Indian Institute of Technology Jodhpur 342037, India. (E-mail: tripathi.5@iitj.ac.in; anoopj@iitj.ac.in). Abhisek K. Behera is with the Department of Electrical Engineering, Indian Institute of Technology Roorkee 247667, India. (E-mail:  abhisek.behera@ee.iitr.ac.in).}}

\begin{document}
\maketitle
%\thispagestyle{empty}
%%%%%%%%%%%%%%%%%%%%%%%%%%%%%%%%%%%%%%%%%%%%%%%%%%%%%%%%%%%%%%%%%%%%%%%%%%%%%%%%%%%%%%%%%

\begin{abstract}
A partially distributed economic dispatch algorithm, which renders optimal value in fixed time with the objective of supplying the load requirement as well as the power transmission losses, is proposed in this paper. The transmission losses are modeled using Kron's $\mathcal{B}-$loss formula, under a standard assumption on the values of $\mathcal{B}-$coefficients. The total power supplied by the generators is subjected to time-varying equality constraints due to time-varying nature of the transmission losses. Using Lyapunov and optimization theory, we rigorously prove the convergence of the proposed algorithm and show that the optimal value of power is reached within a fixed-time, whose upper bound dependents on the values of $\mathcal{B}-$coefficients, parameters characterizing the convexity of the cost functions associated with each generator and the interaction topology among them. Finally, an example is simulated to illustrate the theoretical results. 
%The effect of external disturbances in the proposed algorithm is also investigated. Finally, an example is simulated to illustrate our approach. 
 \end{abstract}


%%%%%%%%%%%%%%%%%%%%%%%%%%%%%%%%%%%%%%%%%%%%%%%%%%%%%%%%%%%%%%%%%%%%%%%%%%%%%%%%%%%%%%%%%%%%%%%%%%%%%%%%%%%%%%%%%%%%%%%%%%%%
\section{Introduction}\label{introduction}
\subsection{Motivation and Literature Survey}
The economic dispatch problem (EDP) has been a celebrated problem in the optimal operation and management of power systems. With the rapid integration of renewable energy sources in the microgrid, solving EDP becomes a challenging task due to the scalability of the power network. To encounter such systems with increased robustness, reliability, and efficiency, the
centralized power generation infrastructure is slowly moving towards a distributed one \cite{yazdanian2014distributed}. The primary goal of EDP in a distributed infrastructure is to seek the minimum value of a collective cost function defined over a network of generators. This led to the requirement of an algorithm that works with renewable energy resources in a distributed manner. In this direction, there exist several approaches in the existing literature; for instance, \cite{wang2018distributed, pourbabak2017novel} discussed consensus-based algorithms; \cite{li2014optimal} described a distributed gradient-based algorithm; \cite{yun2019initialization} studied initialization-free privacy-guaranteed distributed algorithm; \cite{wan2021adaptive} presented a gossip-based distributed algorithm; an adaptive event-triggered distributed algorithm is considered in \cite{wan2021adaptive} etc.

One of the main concerns in designing a distributed algorithm is that it must ensure a faster convergence rate, as the power output changes frequently due to the continuous use of distributed generation systems and dynamic pricing \cite{kang2021distributed}. Addressing these facts, the efforts in existing literature have been towards developing algorithms with finite or fixed-time convergence based on \cite{bhat2000finite,polyakov2011nonlinear}. For example, \cite{chen2016distributed} proposes a distributed finite time algorithm which can address the EDP in the smart grid with/without power generation constraints; \cite{chen2018fixed} presents a distributed continuous-time algorithm to solve a convex optimization problem with equality constraint, which reaches the optimal value in fixed time; \cite{dai2020distributed} extends these results by proposing a new lemma which guarantees finite-time convergence with a tighter upper bound on the convergence time; \cite{baranwal2020robust} discusses user-specified fixed-time consensus-based algorithm to solve EDP with time-varying topology.

In addition to supplying the load demand, it is equally important that the distributed algorithm must satisfy the constraints posed by the time-varying power transmission losses \cite{zhong2013dynamic}. Existing works in this direction primarily consider a simplified model for the transmission losses and discuss asymptotic or exponential convergence to the optimal solution \cite{loia2013decentralized,binetti2014distributed,garcia2019approximating,kouveliotis2017distributed,zhang2015distributed,elsayed2014fully}. Further, \cite{chen2016distributed,chen2018fixed,dai2020distributed,baranwal2020robust} do not address the aspect of transmission losses. Unlike these works, in this paper, we propose an algorithm that accounts for Kron's modeled power transmission losses and reaches the optimal solution of the EDP in a fixed time.

\subsection{Contributions}
Aggregation of the Kron's modeled power transmission losses, by nature, poses an additional requirement of globally sharing the generated power information among the generators. Addressing this fact, the proposed algorithm in the paper considers that the generators have a two-layered communication topology$-$the generated power is shared globally in order to obtain the total power transmission losses, while the cost function and other auxiliary variables are shared locally. Such multi-layered topological considerations are motivated from many works \cite{wen2020distributed,qin2016leaderless,xu2016layered} in this direction in the context of multi-agent systems, deployed for various collaborative missions. Further, our analysis is based on certain assumptions relying on an interplay between the eigenvalues of the matrix $\mathcal{B}$, and parameters characterizing the convexity of the cost functions associated with each generator. The main contributions of this work can be summarized as follows:  
\begin{enumerate}
	\item[i)] We propose a novel consensus-based partially distributed algorithm, which solves the EDP in the presence of power transmission losses characterized by the Kron's $\mathcal{B}$-loss formula \cite{chang1994new}. 
	\item[ii)] Using tools from Lyapunov stability and optimization theory, we rigorously show that the optimal solution of the EDP is rendered in a fixed time. An analytical expression of the upper bound on the convergence time is obtained, which is independent of initial values of power and dependent on the eigenvalues of the Kron's $\mathcal{B}-$loss matrix, the convexity of the cost function associated with each generator and network topology among them.
\end{enumerate}

\subsection{Paper Structure}
The paper unfolds as follows: Section~\ref{section2} describes Kron's transmission loss formula, formulates the problem, and presents some preliminary results on finite-time stability. Section~\ref{section3} derives a few introductory results, describes the proposed algorithm, and obtain an upper bound on the convergence time. Theoretical results are illustrated through a simulation example in Section~\ref{section4}. Finally, Section~\ref{section5} concludes the paper and presents future directions of the work.

\paragraph*{Notations}
Throughout the paper, $\mathbb{R}$ and $\mathbb{R}_+$ denote the set of real and non-negative real numbers, respectively. For any $x \in \mathbb{R}$, we define function $ \sig^\mu: \mathbb{R} \to \mathbb{R}$ as $\sig^\mu (x) = |x|^\mu \text{sign}(x), \mu>0$, where $\text{sign}(x)$ is the signum function of $x$. The Hadamard product (or element-wise product) of two matrices $X$ and $Y$ of the same dimension $m \times n$ is defined as $[(X \odot Y)_{ij}] \coloneqq [X_{ij}][Y_{ij}]$. Let $\psi = [\psi_1, \ldots, \psi_N]^T \in \mathbb{R}^N$, then $\text{diag}\{\psi\}$ denotes the diagonal matrix with the entries of $\psi$ along its principal diagonal. $\nabla f(\bullet)$ and $\nabla^2 f(\bullet)$ represent the gradient and Hessian of the function $f: \mathbb{R}^n \to \mathbb{R}$ with respect to its argument $\bullet$, respectively. The Jacobian of a function $g: \mathbb{R}^n \to \mathbb{R}^m$ is defined to be an $m \times n$ matrix whose $(i, j)^{\text{th}}$ entry is $J _{ij} = \partial g_i/\partial x_j$. We represent by $\pmb{1}_N = [1, \ldots, 1]^T \in \mathbb{R}^N$ and $\pmb{0}_N = [0, \ldots, 0]^T \in \mathbb{R}^N$, respectively. $I_N$ denotes the identity matrix of order $N \times N$. We use symbols $\succeq, \preceq$ to represent element-wise comparison between two matrices of the same size. 

An undirected graph $\mathcal{G} = (\mathcal{V}, \mathcal{E}, \mathcal{A})$ is a collection of node set $\mathcal{V} = \{1, \ldots, N\}$, the edge set $\mathcal{E} \subseteq  \mathcal{V} \times \mathcal{V}$, along with edge weights captured by the adjacency matrix $\mathcal{A} = [a_{ij}] \in \mathbb{R}^{N\times N}$ with $a_{ij} = a_{ji} > 0$ if $(i, j) \in \mathcal{E}$, and $a_{ij} = 0$ otherwise. The Laplacian of $\mathcal{G}$ is defined as $\mathcal{L} = [\ell_{ij}] \in \mathbb{R}^{N \times N}$ with $\ell_{ii} = \sum_{j \in {\mathcal{N}_i}} a_{ij}$ and $\ell_{ij} = -a_{ij},~ \forall i \neq j$, where $\mathcal{N}_i$ is the set of neighboring vertices of vertex $i$. For an undirected and connected graph, $0$ is a simple eigenvalue of $\mathcal{L}$ with the corresponding eigenvector $\pmb{1}_N$, and all the other eigenvalues are positive. %We denote by $\phi_2$ the second smallest eigenvalue of $\mathcal{L}$.

\section{Kron's Formula, Problem Description, and Preliminary Results}\label{section2}
This section reviews the Kron's $\mathcal{B}-$loss formula for power transmission losses, formulates the problem in this paper, and discuss some preliminary results.  

\subsection{Transmission Losses}
Transmission losses in a power system network are often evaluated using Kron's approximated loss Formula. An expression for transmission losses in terms of source loading and a set of loss coefficients (usually referred to as $\mathcal{B}-$coefficients) is of the quadratic form:
\begin{equation}\label{Transmission losses B loss}
P_L = \sum_{i = 1}^N \sum_{j = 1}^N P_i \mathcal{B}_{ij} P_j + \sum_{i = 1}^N P_i\mathcal{B}_{i0} + \mathcal{B}_{00}, 
\end{equation}
where $\mathcal{B}_{ij} $, $\mathcal{B}_{i0}$ and $\mathcal{B}_{00}$ are constant $\mathcal{B}-$loss coefficients and can be evaluated using methods as discussed in \cite{chang1994new, ongsakul2019artificial}. Further, $P_i$ and $P_j$ are the power outputs of generators $i$ and $j$ in megawatts, respectively. The expression \eqref{Transmission losses B loss} can be compactly re-written as $P_L = \sum_{i = 1}^N P_{Li}$, where, 
\begin{equation}\label{trans 1}
P_{Li} = \sum_{j = 1}^N P_i \mathcal{B}_{ij} P_j + P_i\mathcal{B}_{i0} + \mathcal{B}_{{00}_i}, 
\end{equation}
is the power transmission loss associated with the $i^{\text{th}}$ generator and $\mathcal{B}_{00} = \sum_{i = 1}^N \mathcal{B}_{{00}_i}$. 

\subsection{Problem Formulation}
Consider a network comprising $N$ generators in a grid and the cost function of individual generators is given as $C_i(P_i)$. The main objective here is to cooperatively minimize the total cost, that is, the sum of all individual \emph{local} objective functions $C_i(P_i)$, while maintaining an equality constraint, defined in terms of the load demand and power transmission losses $P_L$. Let $D$ and $P_T$ be the total load demand, and total power supplied by the system of generators, respectively. With this description, the economic dispatch problem can be formulated as:
\begin{subequations}\label{opt_problem}
%\begin{equation}\label{conditions 1}
%\text{Min}~\mathcal{C}(P) = \sum_{i = 1}^N C_i(P_i)
%\end{equation}
\begin{align}
\label{conditions 1} &\text{Min}~{C}(P) = \sum_{i = 1}^N C_i(P_i)\\
\nonumber &\text{subject to}~ \sum_{i = 1}^N {P_i} = D + P_L = \sum_{i = 1}^N D_{i} + \sum_{i = 1}^N P_{L_i}\\
\label{conditions 2} & ~~~~~~~~~~~~~~~~~~~~~~~~~~~~~~~~~~= \sum_{i = 1}^N D_{i0} + \sum_{i = 1}^N P_{L_i} = P_T,
\end{align}
\end{subequations}
where $D_{i0}$ is the initial value of the time-varying load demand $D_i$ correspond to the $i^{\text{th}}$ generator-load pair. It is assumed that the load demand is constant at all time, that is, $\sum_{i = 1}^N D_{i} = \sum_{i = 1}^N D_{i0}$ for all $t \geq 0$, which is often a standard assumption in power system networks \cite{chen2018fixed}. It is worth noting that inclusion of power transmission losses $P_L$ does not result in trivially regularizing the overall cost function \eqref{conditions 1}, instead, it affects the equality constraints \eqref{conditions 2} of the optimization problem \eqref{opt_problem} and makes it challenging. Unless otherwise stated, $P_i \geq 0, \forall i$ in our analysis, as the generated power can not be negative.

\subsection{Some Preliminary Results}
Below we describe some useful results that will be helpful in the sequel. 
\begin{lemma}[\hspace{-.1pt}\cite{polyakov2011nonlinear}]\label{Lemma}
	Consider the dynamical system $\dot{x} = f(x(t))$, where $x \in \mathbb{R}^n$, $f: \mathbb{R}^N \to \mathbb{R}^N$ is a continuous function with $f(\pmb{0}_N) = \pmb{0}_N$. Assume that the origin is the equilibrium point of the system. If there exist a continuous radially unbounded Lyapunov function $V: \mathbb{R}^N \to \mathbb{R}_{+} \cup \{0\}$ such that $V(x) = 0 \Leftrightarrow x = 0$ and any solution of $x(t)$ of the system satisfies the inequality $\dot V(x(t)) \leq  -(\alpha V^p(x(t)) + \beta V^q(x(t)))^k$ for some $\alpha, \beta, p, q, k > 0; pk < 1, qk > 1$, then the origin of the system is globally fixed-time stable, and the following estimates of the settling time holds:
	\begin{equation}\label{settling time}
	T_s \leq \frac{1}{{{\alpha ^k}(1 - pk)}} + \frac{1}{{{\beta ^k}(qk - 1)}}.
	\end{equation}
\end{lemma}

\begin{lemma}[\hspace{-.1pt}\cite{zuo2015nonsingular}]\label{Lemma 2}
	Let  $\zeta_i \geq 0$ for $i = \{1, \ldots, N\}$. Then
	\begin{subequations}\label{sub equation 1}
		\begin{align}
		\sum_{i = 1}^N {\zeta_i^m}  & \geq \left(\sum_{i = 1}^N \zeta_i \right)^m,~\text{if}~0 < m \leq 1, \\
		\sum_{i = 1}^N {\zeta_i^m}  & \geq N^{(1-m)} \left(\sum_{i = 1}^N \zeta_i \right)^m,~\text{if}~1 < m < \infty. 
		\end{align}
	\end{subequations}
\end{lemma}

\section{Main Results}\label{section3}
This section presents our main results by proposing an algorithm to solve the optimization problem \eqref{conditions 1} with equality constraints \eqref{conditions 2} in the presence of transmission losses \eqref{Transmission losses B loss}. The proposed algorithm is as follows:
\begin{equation}\label{algorithm}
\left\{
\begin{array}{>{\displaystyle\tallstrut}l@{}}
{\lambda _i} = \frac{{\partial {C_i}({P_i})}}{{\partial {P_i}}}\\ 
\addlinespace
{H_i} = \left( {1 + \frac{{\partial {P_{L_i}}}}{{\partial {P_i}}}} \right)\\
\addlinespace
{{\dot z}_i} =  - {k_1}\sig\left[ \sum\limits_{j \in {\mathcal{N}_i}} {a_{ij}}\left({H_j}{\lambda _j} - {H_i}{\lambda _i}\right)  \right]^\mu \\ ~~~~~ - {k_2} \sig\left[{\sum\limits_{j \in {\mathcal{N}_i}} {{a_{ij}}\left({H_j}{\lambda _j} -{H_i}{\lambda _i}\right)} } \right] ^\nu \\
\addlinespace
{P_i} = \sum\limits_{j \in {\mathcal{N}_i}} {{a_{ij}}({z_j} - {z_i}) + {D_{i0}} + {P_{Li}}}, 
\end{array}
\right.
\end{equation}
where $k_1$, $k_2$, $\mu$ and $\nu$ are constants such that they satisfy conditions $0 < \mu < 1$ and $\nu > 1$; $H_i$, $z_i$ and $\lambda_i$ are intermediate variables. Unlike \cite{chen2016distributed,chen2018fixed}, the algorithm \eqref{algorithm} also accounts for transmission losses $P_L$ by assimilation of an additional term $H_i$, which further influences the dynamics $\dot{z}_i$ of the auxiliary variable $z_i$. Later, we also discuss through simulations that proposed consensus dynamics of auxiliary variables ${z}_i$ can handle a special class of bounded disturbances. 

\begin{remark}
As will be shown in below Lemma~\ref{lemma3.1}, the computation of term $\partial P_{L_i}/\partial P_i$ in algorithm \eqref{algorithm} requires the information of power generated by all the generators. Thereby, the implementation of \eqref{algorithm} requires global topology for obtaining $H_i$, and local topology for sharing the information about the cost function $\lambda_i$ and the auxiliary variables $z_i$ for each $i$. This is the reason we call it a partially distributed consensus algorithm. Although the problem can be solved in a completed distributed way by considering only a single local network for the simplified Korn's modeled transmission losses $P_L = \sum_{i=1}^{N} \mathcal{B}_i P_i^2$ as discussed in \cite{loia2013decentralized,kouveliotis2017distributed}, however, this would be a special case of the problem addressed in our paper. 
\end{remark}

Before proceeding further, we incorporate the following assumptions in our analysis:

\begin{assumption}[Network topology]\label{Assumption 1}
	The generators have a two-layered network topology $-$ the information (only) about generated powers is shared globally among them, and the information about their cost function and other auxiliary variables is shared locally, according to an undirected and connected topology.
\end{assumption}

\begin{assumption}[Cost function]\label{Assumption 2}
For $i = \{1, \ldots, N\}$, the cost function ${C}_i(P_i)$ is a strongly convex function such that $\nabla^2 {C}_i(P_i) \geq \sigma > 0$ for constant $\sigma \in \mathbb{R}_+$. Further, there exists a $\delta \in \mathbb{R} \setminus \{0\}$ such that $\nabla {C}_i(P_i) \geq \delta$, $\forall i$.
\end{assumption}
 
\begin{assumption}[$\mathcal{B}-$loss coefficients]\label{Assumption 3}
The kron's $\mathcal{B}-$loss coefficient matrix $\mathcal{B} = [\mathcal{B}_{ij}], \forall i, j = \{1, \ldots, N\}$ in \eqref{Transmission losses B loss} is symmetrical with all its elements $\mathcal{B}_{ij} \geq 0$ such that  
\begin{enumerate}[label={\emph{(A\arabic*)}}]
\item $0 \leq {{\partial {P_{L_i}}}}/{{\partial {P_i}}} < 1, \forall i= \{1, \ldots, N\}$. 
\item Let $b_1 \leq b_2 \leq \cdots \leq b_N$ be the eigenvalues of $\mathcal{B}$. Denote by $\rho = \min_{i}\{\mathcal{B}_{i0}\}$. Then, the parameters $\sigma, \delta$ and $\rho$ are such that they satisfy: $(1 + \rho)\sigma + b_1 \delta > 0, \ \text{if} \ \delta > 0$, and $(1 + \rho)\sigma + b_N \delta > 0, \ \text{if} \ \delta < 0$.
%\begin{align*}
%(1 + \rho)\sigma + b_1 \delta > 0, \ \text{if} \ \delta > 0;\\
%(1 + \rho)\sigma + b_N \delta > 0, \ \text{if} \ \delta < 0.
%\end{align*} 
\end{enumerate}
\end{assumption}

\begin{remark}\label{B_coeff_justification}
It is to be noted that Assumption~\emph{(A1)} is common for practical power system networks (for instance, please refer to \cite{ongsakul2019artificial,fromm1985evaluation,moon2002slack,el1995electrical}). This is due to the fact that the values of $\mathcal{B}$-loss coefficients are usually very small such that total transmission losses $P_L$ are negligible compared to the value of total load demand $D$. Following this, one can write from \eqref{conditions 2} that $\sum_{i = 1}^N {P_i} \approx \sum_{i = 1}^N D_{i0} = \bar{D}$ (say) for small values of $\mathcal{B}$-coefficient. This implies that ${P_i} \leq \bar{D}$ for each $i$. According to Assumption~\emph{(A1)}, it follows from \eqref{trans 1} that the inequality ${{\partial {P_{L_i}}}}/{{\partial {P_i}}} = \sum_{j = 1,j\neq i}^N \mathcal{B}_{ij} P_j + 2\mathcal{B}_{ii}P_i + \mathcal{B}_{i0} < 1$ must hold true for all $t \geq 0$. For the given $\bar{D}$, this can be assured only if the $\mathcal{B}-$coefficients are such that $\sum_{j = 1, j\neq i}^N \mathcal{B}_{ij} + 2 \mathcal{B}_{ii} + \mathcal{B}_{i0}{\bar{D}}^{-1} < {\bar{D}}^{-1}$ for each $i$. In fact, the term ${{\partial {P_{L_i}}}}/{{\partial {P_i}}}$ can be obtained from the well known notion of penalty factor, defined by $1/(1 - ({{\partial {P_{L_i}}}}/{{\partial {P_i}}}))$ in the literature \cite{fromm1985evaluation,moon2002slack,el1995electrical}, and justifies our assumption. Obviously, $\partial P_L/\partial P_i \geq 0$, as $P_i \geq 0$ for all $i$ and $t \geq 0$. 
\end{remark}

We now discuss the following lemmas before stating the main result.  

\begin{lemma}\label{lemma3.1}
Under Assumption~\ref{Assumption 3}, the following relation holds:  
\begin{equation}\label{eq r3}
\frac{{\partial {P_{L_i}}}}{{\partial {P_i}}}= \frac{{1}}{{2}} \frac{{\partial {P_{L}}}}{{\partial {P_i}}}+\mathcal{B}_{ii}P_i+\frac{{\mathcal{B}_{i0}}}{{2}}, \ \forall i.
\end{equation}
\end{lemma}

\begin{proof}
Differentiating \eqref{Transmission losses B loss} and \eqref{trans 1} with respect to $P_i$, we have
\begin{align}
\label{eq r1} \frac{{\partial {P_{L}}}}{{\partial {P_i}}} &=  2\sum\limits_{j = 1,j\neq i}^N \mathcal{B}_{ij} P_j  + 2\mathcal{B}_{ii}P_i  + \mathcal{B}_{i0},\\
\label{eq r2} \frac{{\partial {P_{L_i}}}}{{\partial {P_i}}} &=  \sum\limits_{j = 1,j\neq i}^N \mathcal{B}_{ij} P_j  + 2\mathcal{B}_{ii}P_i  + \mathcal{B}_{i0},
\end{align}
which leads to required result, as $[\mathcal{B}_{ij}] = [\mathcal{B}_{ji}], \forall i, j$.
\end{proof}

\begin{lemma}\label{Lemma3}
Let $P = [P_1, \ldots, P_N]^T$ be the vector of all generator bus net outputs. Define $\mathcal{F} = \nabla P_L(P)\odot\nabla \mathcal{C}(P), \mathcal{M} = \nabla\mathcal{R}(P)\odot\nabla \mathcal{C}(P)$, and $\mathcal{Q} = \nabla^2 \mathcal{C}(P) + 0.5 \nabla \mathcal{F} + \nabla \mathcal{M}$, where,
\begin{equation*}
\nabla\mathcal{R}(P) = \left[\left(\mathcal{B}_{11}P_1 + \frac{\mathcal{B}_{10}}{2}\right), \ldots, \left(\mathcal{B}_{NN}P_N+\frac{\mathcal{B}_{N0}}{2}\right)\right]^T.
\end{equation*}
Further, let $\mathcal{S}$ be an $N \times N$ matrix with diagonal entries $[\mathcal{S}_{ii}] = (1 + \mathcal{B}_{i0})\sigma + 2 \mathcal{B}_{ii}\delta$ and off-diagonal entries $[S_{ij}] = \mathcal{B}_{ij}\delta$. Under Assumptions~\ref{Assumption 2} and \ref{Assumption 3}, the following properties hold:
\begin{enumerate}[label={\emph{(R\arabic*)}}]
	\item $[\nabla \mathcal{F}_{ii}] \geq  2\mathcal{B}_{ii}\delta+\mathcal{B}_{i0}\sigma$ and $[\nabla \mathcal{F}_{ij}] \geq  2\mathcal{B}_{ij}\delta$.
	\item $\nabla \mathcal{M}$ is a diagonal matrix with $[\nabla \mathcal{M}_{ii}] \geq \left(\mathcal{B}_{ii}\delta + \frac{B_{i0}}{2}\sigma\right)$.
	\item $[\mathcal{Q}_{ii}] \geq \sigma + 2 \mathcal{B}_{ii}\delta + \mathcal{B}_{i0}\sigma$ and $[\mathcal{Q}_{ij}] \geq \mathcal{B}_{ij}\delta$.
	\item $\mathcal{S}$ is a symmetric matrix satisfying $\mathcal{S} \preccurlyeq \mathcal{Q}$.
	\item Let $\tau_1 \leq \tau_2 \leq \cdots \leq \tau_N$ be the eigenvalues of $\mathcal{S}$. Then, $b_1 \delta \leq \tau_i - \sigma(1 + \mathcal{B}_{i0}) \leq b_N \delta$, if $\delta > 0$; and $b_N \delta \leq \tau_i - \sigma(1 + \mathcal{B}_{i0}) \leq b_1 \delta$, if $\delta < 0$, for each $i$, where $b_1$ and $b_N$ are the smallest and largest eigenvalues of $\mathcal{B}$, as defined in Assumption~\emph{(A2)}.
\end{enumerate}
\end{lemma}

Please refer to Appendix for the proof.  We are now ready to state the main result: 

\begin{thm}\label{thm1}
The algorithm \eqref{algorithm}, under the Assumptions~\ref{Assumption 1}, \ref{Assumption 2} and \ref{Assumption 3}, solves the economic load dispatch problem \eqref{opt_problem} in a fixed time.
\end{thm}

\begin{proof}\label{proof}
The sum of power supplied by each generator at any time instant satisfies
\begin{align}
\nonumber \sum_{i=1}^{N}{{{P}_{i}}}& = \sum_{i=1}^{N}{\left(\sum_{j \in \mathcal{N}_i} {{{a}_{ij}}({{z}_{j}}-{{z}_{i}}) + D_{i0} + P_{L_i}} \right)}\\
\label{sum P} & = \sum\limits_{i = 1}^N D_{i0} + \sum\limits_{i = 1}^N P_{L_i} = P_T,  
\end{align}
as $\sum_{i=1}^{N} \sum_{j \in \mathcal{N}_j} a_{ij}({z_j} - {z_i}) = 0$ for an undirected and connected graph with $a_{ij} = a_{ji}$. Clearly, \eqref{sum P} satisfies the desired equality constraint \eqref{conditions 1}. Substituting for $P_i$ from \eqref{algorithm}, the optimization problem \eqref{conditions 1} can be represented as the following  unconstrained optimization problem:
\begin{equation}\label{min c}
\text{Min } {C}(z) = \sum_{i = 1}^N {{{C}_i}\left( {\sum\limits_{j \in {N_i}} {{a_{ij}}({z_j} - {z_i}) + D_{i0} + P_{Li}} } \right)}. 
\end{equation}
From \eqref{algorithm}, the derivative of ${P_i}$ with respect to ${z_j}$ is obtained as:
\begin{align}
\label{partail_derivative} \frac{{\partial {P_i}}}{{\partial {z_j}}} &= 
        \begin{cases}
            - \sum_{j = 1}^N {{a_{ij}}}  + \frac{{\partial {P_{L_i}}}}{{\partial {z_j}}} & \text{if} \ j = i\\
            {a_{ij}} + \frac{{\partial {P_{L_i}}}}{{\partial {z_j}}} & \text{if} \ j \neq i
        \end{cases}\\
       \label{partial_der_1} &= 
       \begin{cases}
       - \sum_{j = 1}^N {{a_{ij}}}  + \frac{{\partial {P_{L_i}}}}{{\partial {P_i}}}\frac{{\partial {P_{i}}}}{{\partial {z_j}}} & \text{if} \ j = i\\
       {a_{ij}} + \frac{{\partial {P_{L_i}}}}{{\partial {P_i}}}\frac{{\partial {P_{i}}}}{{\partial {z_j}}} & \text{if} \ j \neq i.
       \end{cases} 
    \end{align}
Once again using \eqref{partail_derivative} in \eqref{partial_der_1} for $\frac{{\partial {P_{i}}}}{{\partial {z_j}}}$, we have
\begin{equation}\label{partial der 2}
 \frac{{\partial {P_i}}}{{\partial {z_j}}}   = 
        \begin{cases}
            - \sum_{j = 1}^N {{a_{ij}}} + \frac{{\partial {P_{L_i}}}}{{\partial {P_i}}} \\
            \ \times \left[ - \sum_{j = 1}^N {{a_{ij}}}  + \frac{{\partial {P_{Li}}}}{{\partial {P_i}}}\frac{{\partial {P_{i}}}}{{\partial {z_j}}}\right] &  \text{if} \ j = i\\
            {a_{ij}} + \frac{{\partial {P_{Li}}}}{{\partial {P_i}}}\left[ {{a_{ij}}}  + \frac{{\partial {P_{L_i}}}}{{\partial {P_i}}}\frac{{\partial {P_{i}}}}{{\partial {z_j}}}\right] & \text{if} \ j \neq i.
       \end{cases}
    \end{equation}
Continuing the substitution in each step, an infinite series is formed for ${{\partial {P_i}}}/{{\partial {z_j}}}$, as below:
\begin{equation}\label{partial der 3}
 \frac{{\partial {P_i}}}{{\partial {z_j}}} =
 \begin{cases}
            - \sum_{j = 1}^{N} {{a_{ij}}}  - \sum_{j = 1}^N {{a_{ij}}}  \frac{{\partial {P_{L_i}}}}{{\partial {P_i}}} - \\
            \ \sum_{j = 1}^N {{a_{ij}}} \left(\frac{{\partial {P_{L_i}}}}{{\partial {P_i}}}\right)^2 - \cdots   & j = i\\
            {a_{ij}} + {a_{ij}}\frac{{\partial {P_{Li}}}}{{\partial {P_i}}}+{a_{ij}}\left(\frac{{\partial {P_{L_i}}}}{{\partial {P_i}}}\right)^2 + \cdots & j \neq i.
        \end{cases}
  \end{equation}
Following Assumption~{(A1)}, the higher order terms are neglected to get:
\begin{equation}\label{partial der 3.1}
         \frac{{\partial {P_i}}}{{\partial {z_j}}} =  
       \begin{cases}
            - \sum_{j = 1}^N {{a_{ij}}}\left(1+ \frac{{\partial {P_{Li}}}}{{\partial {P_i}}}\right)  & \text{if} \ j = i\\
            {a_{ij}}\left(1+ \frac{{\partial {P_{Li}}}}{{\partial {P_i}}}\right)  & \text{if} \ j \neq i,
        \end{cases}
    \end{equation}
which on substitution for $\partial P_{Li}/\partial P_i$ from Lemma \ref{lemma3.1} results in
\begin{equation}\label{partial der 3.2}
\frac{{\partial {P_i}}}{{\partial {z_j}}}   =  
\begin{cases}
- \sum\limits_{j = 1}^N {{a_{ij}}}  \left(1 +\frac{{1}}{{2}} \frac{{\partial {P_{L}}}}{{\partial {P_i}}}+\mathcal{B}_{ii}P_i+\frac{{\mathcal{B}_{i0}}}{{2}}\right)  &  \text{if} \ j = i\\
{a_{ij}}\left(1+ \frac{{1}}{{2}} \frac{{\partial {P_{L}}}}{{\partial {P_i}}}+\mathcal{B}_{ii}P_i+\frac{{\mathcal{B}_{i0}}}{{2}}\right) &  \text{if} \ j \neq i.
\end{cases}
\end{equation}
Similarly, it can be written from \eqref{algorithm} about the cost function that:
\begin{align}
       \nonumber & \frac{{\partial {C_i}}}{{\partial {z_j}}} = \frac{{\partial {C_i}}}{{\partial {P_i}}}\frac{{\partial {P_i}}}{{\partial {z_j}}}\\
\label{partial der 5} & =  
     \begin{cases}
     - \sum\limits_{j = 1}^N {{a_{ij}}}  \left(1 +\frac{{1}}{{2}} \frac{{\partial {P_{L}}}}{{\partial {P_i}}}+\mathcal{B}_{ii}P_i+\frac{{\mathcal{B}_{i0}}}{{2}}\right)\lambda_i  &  \text{if} \ j = i\\
     {a_{ij}}\left(1+ \frac{{1}}{{2}} \frac{{\partial {P_{L}}}}{{\partial {P_i}}}+\mathcal{B}_{ii}P_i+\frac{{\mathcal{B}_{i0}}}{{2}}\right)\lambda_i &  \text{if} \ j \neq i.
     \end{cases}   
    \end{align}
Note that the gradient of $P_L$ is $\nabla P_L(P) = \left[\frac{\partial P_L}{\partial P_1}, \ldots, \frac{\partial P_L}{\partial P_N}\right]^T$, using which, \eqref{partial der 3.2} can be expressed in the form of Jacobian as
\begin{align}
\nonumber {\mathcal{J}}_{P} &= {\displaystyle {\frac {\partial (P_{1}, \ldots,P_{N})}{\partial (z_{1}, \ldots,z_{N})}}}\\ 
\label{gradient.1} &= - (I_N + 0.5~\text{diag}\{\nabla P_L(P)\}+\text{diag}\{\nabla \mathcal{R}(P)\})\mathcal{L},
\end{align}
where $\nabla \mathcal{R}(P)$ is defined in Lemma~\ref{Lemma3} and $\mathcal{L}$ is the Laplacian of the underlying topology. Further, the gradient of ${C}(z)$, using \eqref{partial der 5}, is given by
\begin{align}\label{gradient}
\nabla {C}(z) =  - \mathcal{L} \bigg[\bigg(I_N+0.5&~\text{diag}\{\nabla P_L(P)\}&&\notag\\
 & + \text{diag} \{\nabla\mathcal{R}(P)\}\bigg)  \nabla {C}(P)\bigg].&&
\end{align}
We emphasize here that $\mathcal{J}_p$ is $N \times N$ matrix, while $\nabla {C}(z)$ is an $N \times 1$ vector, as $\nabla C(P) = \left[\frac{\partial C_1}{\partial P_1}, \ldots, \frac{\partial C_N}{\partial P_N}\right]^T$. The Hessian of \eqref{gradient} satisfies,
\begin{align}\label{equation123.1}
 \nonumber \nabla^2 {C}(z) =  -\mathcal{L}\nabla \bigg[\bigg(I_N + 0.5~& \text{diag}(\nabla P_L(P))&&\\
 &+\text{diag}(\nabla\mathcal{R}(P))\bigg)  \nabla {C}(P)\bigg]{\mathcal{J}}_{P},
\end{align}
which, further simplifying the term inside the square bracket and using \eqref{gradient.1}, yields
\begin{align}
 \nabla^2 {C}(z) &=  \mathcal{L} \bigg[\nabla^2 {C}(P)+0.5~\nabla (\nabla P_L(P)\odot\nabla {C}(P)) &&\notag\\
 & +\nabla (\nabla\mathcal{R}(P) \odot  \nabla {C}(P))\bigg] \times &&\notag\\
 & (I_N + 0.5~\text{diag}\{\nabla P_L(P)\}+\text{diag}\{\nabla \mathcal{R}(P)\})\mathcal{L}.
\end{align}
From Lemma~\ref{Lemma3}, please note that $\nabla P_L(P)\odot\nabla {C}(P) = \mathcal{F}$ and $\nabla\mathcal{R}(P)\odot \nabla {C}(P) = \mathcal{M}$, which implies that
\begin{align}\label{test1}
\nabla^2 {C}(z) & =  \mathcal{L} (\nabla^2 {C}(P)+0.5~ \nabla\mathcal{F}  + \nabla\mathcal{M})&&\notag \\ 
& \times (I_N + 0.5~\text{diag}\{\nabla P_L(P)\}+\text{diag}\{\nabla \mathcal{R}(P)\})\mathcal{L}. 
\end{align}
Let the optimal solution of convex optimization problem \eqref{min c} be given as $z^* = [{z}^\star_1, \ldots, {z}^\star_N]$. The trivial solution is given by $z^*\in \beta\pmb{1}_N$, where constant $\beta \in \mathbb{R}$. The focus of our analysis is on non-trivial case where solutions belong to the convex and compact set $\mathcal{Z} \subset \mathbb{R}^N \setminus \beta \pmb{1}_N$. For any $z, \xi \in \mathcal{Z}$, it follows for the strongly convex functions from \cite{BoydS2004} that,
\begin{equation}\label{gradient 3}
\mathcal{C}(z) = {C}(\xi)+\nabla^T{C}(\xi)(z-\xi) + \frac{1}{2} (z-\xi)^T\nabla^2 {C}(\hat{z})(z-\xi),
\end{equation}
where $\hat{z} = \xi + \eta(z-\xi)$ with $\eta \in [0, 1]$. Replacing $z, \xi$ by $z^\star, z$, respectively, \eqref{gradient 3} becomes
\begin{equation}\label{eq 2}
{C}(z^\star) = {C}(z) + \nabla^T {C}(z)(z^\star - z) + \frac{1}{2} (z^\star-z)^T\nabla^2 {C}(\tilde{z})(z^*-z),
\end{equation}
where $\tilde{z} = z + \eta(z^\star-z)$ with $\eta \in [0, 1]$. Rearranging \eqref{eq 2} as
\begin{equation*}
{C}(z)-{C}(z^\star) =\nabla^T {C}(z)(z-z^\star)- \frac{1}{2} (z^\star-z)^T\nabla^2 {C}(\tilde{z})(z^\star-z), 
\end{equation*}
and using Assumption~\ref{Assumption 2}, it holds that
\begin{align}\label{eq 4}
{C}(z)-\mathcal{C}(z^\star) \leq \nabla^T {C}(z)(z - z^\star).
\end{align}
Let  $\phi_1,\phi_2,\cdots,\phi_N$ be the eigenvalues of Laplacian $\mathcal{L}$ such that $0 = \phi_1 \leq \phi_2 \leq\ldots \leq \phi_N$ with corresponding orthogonal eigenvectors $\pmb{1}_N, v_2, \ldots, v_N$, where $||v_i|| = 1, \ i = 2, \ldots, N$. The vector $z-z^\star$ can be expressed as
\begin{align}\label{eq 1}
z-z^\star = \kappa_1 \pmb{1}_N + \kappa_2 v_2 + \ldots +\kappa_N v_N,
\end{align}
where $\kappa_i, i = \{1, \ldots, N\}$ are constants. Using \eqref{eq 1}, \eqref{eq 4} becomes 
\begin{align}\label{eq 5}
{C}(z)-{C}(z^\star) \leq \nabla^T{C}(z)(\kappa_1\pmb{1}_N + v), 
\end{align}
where $v = \kappa_2 v_2 + \ldots + \kappa_N v_n$. Note that $||v||^2 = \kappa^2_2 + \ldots + \kappa^2_N$. From \eqref{gradient} and \eqref{eq 5}, it follows that 
\begin{align}\label{gradient 6}
 \nonumber {C}(z)-\mathcal{C}(z^\star) &\leq  -\kappa_1  \left[\left(I_N + \text{diag}\{\nabla P_L(P)\}\right)\nabla {C}(P)\right]^T \mathcal{L}\pmb{1}_N \\
\nonumber  & ~~~~~~~~~~~ + \nabla^T {C}(z)v\\
& = \nabla^T {C}(z)v \leq ||\nabla^T {C}(z)||||v||,
\end{align}
as $\mathcal{L} = \mathcal{L}^T$ and $\mathcal{L}\pmb{1}_N = \pmb{0}_N$ for an undirected and connected graph. Now, substituting $\xi = z^\star$ in \eqref{gradient 3} and noting that $\nabla C(z^\star) = \pmb{0}_N$, we have
\begin{align}\label{gradient 7}
 {C}(z)-{C}(z^\star) = 0.5(z-z^\star)^T \nabla^2 {C}(\hat{z})(z-z^\star),
\end{align}
which, upon substitution from \eqref{test1}, gives
\begin{align}\label{gradient 8}
&{C}(z) - {C}(z^*) = 0.5 (\mathcal{L} (z-z^\star))^T ( \nabla^2 {C}(P) +0.5~ \nabla \mathcal{F}+\nabla \mathcal{M}) &&\notag\\ 
& \times  (I_N + 0.5~\text{diag}\{\nabla P_L(P)\}+\text{diag}\{\nabla \mathcal{R}(P)\})(\mathcal{L} (z-z^\star)).
\end{align}
According to Assumption~\ref{Assumption 3}, and Lemmas~\ref{lemma3.1} and \ref{Lemma3}, it is clear that $1+0.5~\nabla P_L(P_i)+\nabla \mathcal{R}(P_i) = 1 + \sum_{j = 1,j\neq i}^N \mathcal{B}_{ij} P_j  + 2\mathcal{B}_{ii}P_i + \mathcal{B}_{i0} \geq (1 + \mathcal{B}_{i0}) \geq (1+\rho)$ for each $i$, where $\rho = \min_{i}\{\mathcal{B}_{i0}\}$. This implies that the diagonal matrix $I_N + 0.5~\text{diag}\{\nabla P_L(P)\}+\text{diag}\{\nabla\mathcal{R}(P)\} \succeq  (1+\rho)I_N$. Consequently, it holds from \eqref{gradient 8} that,
 \begin{align}\label{gradient 8.001}
 {C}(z) - {C}(z^*) &\geq 0.5(1+\rho) (\mathcal{L} (z-z^\star))^T &&\notag\\
 &\times ( \nabla^2 {C}(P) +0.5~ \nabla \mathcal{F}+\nabla \mathcal{M})(\mathcal{L} (z-z^\star)) &&\notag\\
& = 0.5(1+\rho) (\mathcal{L} (z-z^\star))^T \mathcal{Q}(\mathcal{L} (z-z^\star)),
\end{align}
where $\mathcal{Q} = \nabla^2 {C}(P) +0.5~ \nabla \mathcal{F}+\nabla \mathcal{M}$, according to Lemma~\ref{Lemma3}. Further, one can write using result (R4) from Lemma~\ref{Lemma3} that:
 \begin{align}\label{A1}
  {C}(z) - {C}(z^*) \geq 0.5(1+\rho) (\mathcal{L} (z-z^\star))^T \mathcal{S}(\mathcal{L} (z-z^\star)),
\end{align}
where $\mathcal{S}$ is a symmetric matrix with eigenvalues $\tau_1 \leq \tau_2\leq \cdots \leq \tau_N$, as per Lemma~\ref{Lemma3}. Using Courant-Fischer theorem [\cite{horn2012matrix}, Chapter~4, pg. 236] for the symmetric matrix $\mathcal{S}$, it holds for $z \in \mathcal{Z}$ that
 \begin{align}\label{A}
&{C}(z) - {C}(z^*) \geq 0.5(1+\rho) \tau_1 (\mathcal{L} (z-z^\star))^T (\mathcal{L} (z-z^\star)) \notag\\
&\geq 0.5(1+\rho) \tau_1 (\kappa_2 \phi_2 v_2 +\kappa_3 \phi_3 v_3+\cdots+\kappa_N \phi_N v_N)^T\notag\\& \times (\kappa_2 \phi_2 v_2 +\kappa_3 \phi_3 v_3+\cdots+\kappa_N \phi_N v_N)\notag\\
&= 0.5 (1+\rho) \tau_1 ({\kappa_2}^2 {\phi_2}^2  +{\kappa_3}^2 {\phi_3}^2 +\cdots+{\kappa_N}^2 {\phi_N}^2 )\notag\\
&\geq  0.5 (1+\rho) \tau_1 \phi^2_2 ||v||^2,
\end{align}
using \eqref{eq 1}. Now, it follows from \eqref{gradient 6} and \eqref{A} that
\begin{align}\label{equn 3}
\nonumber ||\nabla^T {C}(z)||^2||v||^2 &\geq 0.5 (1+\rho) \tau_1 \phi_2^2 ||v||^2( {C}(z)- {C}(z^\star))\\
\implies ||\nabla^T {C}(z)||^2 & \geq 0.5 (1+\rho) \tau_1 \phi_2^2 ( {C}(z)- {C}(z^*)),
\end{align}
as $||v|| \neq 0$ for non-trivial optimal solution. It is worth noticing that $\tau_i > 0, \forall i$, under Assumption~(A2). This follows from the fact (R5) in Lemma~\ref{Lemma3} that $\tau_i  \geq (1+\mathcal{B}_{i0})\sigma + b_1 \delta \geq (1+\rho)\sigma + b_1 \delta > 0$ if $\delta > 0$; and $\tau_i  \geq (1+\mathcal{B}_{i0})\sigma + b_N \delta \geq (1+\rho)\sigma + b_N \delta > 0$ if $\delta < 0$, as per Assumption~(A2).  

Next, we consider the candidate Lyapunov function
\begin{equation}\label{equn 5}
V = 0.5({C}(z)-{C}(z^\star))^2,
\end{equation}
whose time derivative is
\begin{equation}\label{equn 6}
\dot V = ({C}(z)-{C}(z^\star))\nabla^T{C}(z)\dot{z}.
\end{equation}
Using \eqref{algorithm} and \eqref{partial der 5}, it yields that
\begin{align}
\dot z_i &=  - {k_1}  \sig{\left( {\frac{{\partial {C_i}}}{{\partial {z_i}}} + \sum\limits_{j = 1,j \ne i}^N {\frac{{\partial {C_i}}}{{\partial {z_i}}}} } \right)^\mu } &&\notag\\
& \ \ \ \ \ \ \ \ \ \ - {k_2} \sig{\left( {\frac{{\partial {C_i}}}{{\partial {z_i}}} + \sum\limits_{j = 1,j \ne i}^N {\frac{{\partial {C_i}}}{{\partial {z_i}}}} } \right)^\nu }  &&\notag\\
& =  - {k_1} \sig{\left( {\frac{{\partial C}}{{\partial {z_i}}}} \right)^\mu } - {k_2} \sig{\left( {\frac{{\partial C}}{{\partial {z_i}}}} \right)^\nu}, \ \forall i.
\end{align}
It can be rewritten in vector notations that 
\begin{equation*}
\dot z =  - {k_1}\sig{(\nabla {C}(z))^\mu } - {k_2}\sig{(\nabla {C}(z))^\nu},
\end{equation*}
which on substitution in \eqref{equn 6}, yields 
\begin{equation*}
\dot V = ({C}(z)-{C}(z^\star))\nabla^T{C}(z)( - {k_1}\sig{(\nabla {C}(z))^\mu }  - {k_2}\sig{(\nabla {C}(z))^\nu }).
\end{equation*}
Note that \cite{chen2018fixed}
\begin{align*}
\nabla^T {C}(z) \sig{(\nabla {C}(z))^\mu} &= \sum_{i = 1}^N {{{\left|{\frac{{\partial {C}}}{{\partial {z_i}}}} \right|}^{(\mu  + 1)}}} = \sum_{i = 1}^N {{{\left( {\frac{{\partial {C}}}{{\partial {z_i}}}} \right)}^{2\frac{(\mu  + 1)}{2}}}}\\
\nabla^T {C}(z) \sig{(\nabla {C}(z))^\nu } &= \sum_{i = 1}^N {{{\left|{\frac{{\partial {C}}}{{\partial {z_i}}}} \right|}^{(\nu  + 1)}}} = \sum_{i = 1}^N {{{\left( {\frac{{\partial {C}}}{{\partial {z_i}}}} \right)}^{2\frac{(\nu  + 1)}{2}}}}.
\end{align*}
Using these relations, we have
\begin{align}\label{equn 14}
\dot V &= ({C}(z)-{C}(z^\star)) \times &&\notag\\
& \ \bigg[-k_1\sum\limits_{i = 1}^N {{{\left( {\frac{{\partial {C}}}{{\partial {z_i}}}} \right)}^{2\frac{(\mu  + 1)}{2}}}} - k_2  \sum\limits_{i = 1}^N {{{\left( {\frac{{\partial {C}}}{{\partial {z_i}}}} \right)}^{2\frac{(\nu  + 1)}{2}}}} \bigg].
\end{align}
From Lemma~\ref{Lemma 2}, 
\begin{align*}
\sum\limits_{i = 1}^N \left(\frac{\partial C}{\partial z_i} \right)^{2\left(\frac{1 + \mu}{2}\right)} & \geq \left( \sum\limits_{i = 1}^N \left(\frac{\partial C}{\partial z_i} \right)^2 \right)^{\frac{1 + \mu}{2}}\\
\sum\limits_{i = 1}^N \left(\frac{\partial C}{\partial z_i} \right)^{2\left(\frac{1 + \nu}{2} \right)} & \geq N^{\frac{1-\nu}{2}} {{\left( {\sum\limits_{i = 1}^N {{{\left( {\frac{{\partial {C}}}{{\partial {z_i}}}} \right)}^2}} } \right)}^{\frac{{1 + \nu }}{2}}},
\end{align*}
implying that
\begin{align*}
\dot V \leq & -{k_1}({C}(z)-{C}(z^\star)) {{\left( {\sum\limits_{i = 1}^N {{{\left( {\frac{{\partial {C}}}{{\partial {z_i}}}} \right)}^2}} } \right)}^{\frac{{1 + \mu }}{2}}}&&\\
 & - {k_2}({C}(z)-{C}(z^*)) N^{\frac{1-\nu}{2}}{{\left( {\sum\limits_{i = 1}^N {{{\left( {\frac{{\partial {C}}}{{\partial {z_i}}}} \right)}^2}} } \right)}^{\frac{{1 + \nu }}{2}}}\\
\leq & -{k_1}({C}(z)-{C}(z^*)) (||\nabla {C}(z)||^2)^{\frac{{1 + \mu }}{2}}\\
		& - {k_2}({C}(z)-{C}(z^*)) N^{\frac{1-\nu}{2}}(||\nabla {C}(z)||^2)^{\frac{{1 + \nu }}{2}}. 
\end{align*}
Using \eqref{equn 3} and \eqref{equn 5}, we have
\begin{align*}
&\dot V  \leq  -{k_1}({C}(z)-{C}(z^\star)) [0.5 (1+\rho)\tau_1 \phi_2^2 ({C}(z)-{C}(z^\star))]^{\frac{1+\mu}{2}}\\
 & - {k_2} N^{\frac{1-\nu}{2}}({C}(z)-{C}(z^*)) [0.5 (1+\rho)\tau_1 \phi_2^2 ({C}(z)-{C}(z^\star))]^{\frac{1+\nu}{2}}\\
& \leq  -{k_1}(0.5 (1+\rho) \tau_1 \phi_2^2)^{\frac{1+\mu}{2}}2^{{\frac{3+\mu}{4}}} (V)^{{\frac{3+\mu}{4}}} \\
&~~~~~~~~ -{k_2} N^{\frac{1-\nu}{2}}(0.5 (1+\rho) \tau_1 \phi_2^2)^{\frac{1+\nu}{2}}2^{{\frac{3+\nu}{4}}}(V)^{{\frac{3+\nu}{4}}}\\
& \leq  -{k_1}((1+\rho) \tau_1 \phi_2^2)^{\frac{1+\mu}{2}}2^{{\frac{1-\mu}{4}}} (V)^{{\frac{3+\mu}{4}}}\\
&~~~~~~~~ - {k_2} N^{\frac{1-\nu}{2}}((1+\rho) \tau_1 \phi_2^2)^{\frac{1+\nu}{2}}2^{{\frac{1-\nu}{4}}}(V)^{{\frac{3+\nu}{4}}}.
\end{align*}
Following Lemma~\ref{Lemma}, it can be concluded that $\dot V\leq-(\alpha V^p+\beta V^q)$ with $k=1, \alpha={k_1}((1+\rho) \tau_1 \phi_2^2)^{\frac{1+\mu}{2}}2^{{\frac{1-\mu}{4}}}, \beta={k_2} N^{\frac{1-\nu}{2}}((1+\rho) \tau_1 \phi_2^2)^{\frac{1+\nu}{2}}2^{{\frac{1-\nu}{4}}}$, and $ 0 < p={{\frac{3+\mu}{4}}} < 1, q = {{\frac{3+\nu}{4}}} > 1$, and hence, the settling time is bounded by
\begin{align}\label{equn 22} 
\nonumber T_s &\le \frac{4}{{k_1}((1+\rho) \tau_1 \phi_2^2)^{\frac{1+\mu}{2}}2^{{\frac{1-\mu}{4}}}(1-\mu)} \\
&~~~~~~~~ + \frac{4}{{k_2} N^{\frac{1-\nu}{2}}((1+\rho) \tau_1 \phi_2^2)^{\frac{1+\nu}{2}}2^{{\frac{1-\nu}{4}}}(\nu-1)}.
\end{align}
This implies that $V \to 0$ as $t \to T_s$, and hence it follows from \eqref{gradient 7} and \eqref{equn 5} that $z = z^\star$ as $t \geq T_s$. This concludes the proof.
\end{proof}

\begin{remark}
Note that the right-side of the inequality \eqref{equn 22} is well-defined as $\tau_1 > 0$, under Assumption~\emph{(A2)}. According to Lemma~\ref{Lemma3}, since $\tau_1$ depends upon the eigenvalues of matrix $\mathcal{B}$, coefficients $\mathcal{B}_{i0}$, and the constants $\sigma, \delta$ associated with the cost functions $C_i$, the settling time $T_s$ shows a dependence on these parameters, and network topology because of the occurrence of $\phi_2$ (the second smallest eigenvalue of the Laplacian $\mathcal{L}$ associated with the local network). In fact, the inequality \eqref{equn 22} provides an estimate of the upper bound for the convergence time. Its value is robust to the changes in the initial conditions and power transmission losses. The actual convergence time may be much less than this estimated value.	
\end{remark}

\begin{figure}[t]
	\centering 
	\includegraphics[width=0.47\textwidth]{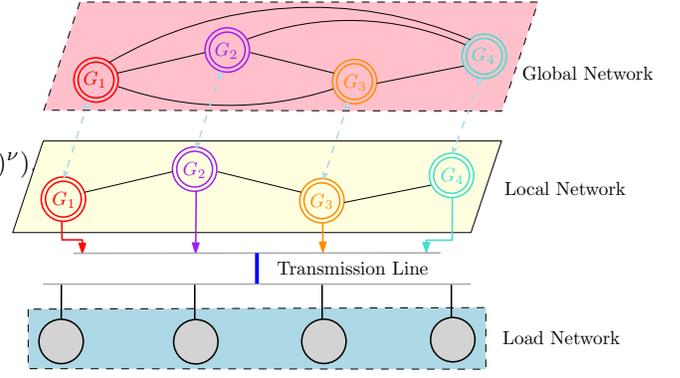}
	\caption{Generator-load configuration with a two-layered interaction topology. Global network is used for sharing the generated power, while the local network for cost function and auxiliary variable in algorithm \eqref{algorithm}.}
	\label{network}
\end{figure}

\begin{figure*}[t]
	\centering 
	\subfigure[$P_i$ Vs time]{\includegraphics[width=0.34\textwidth]{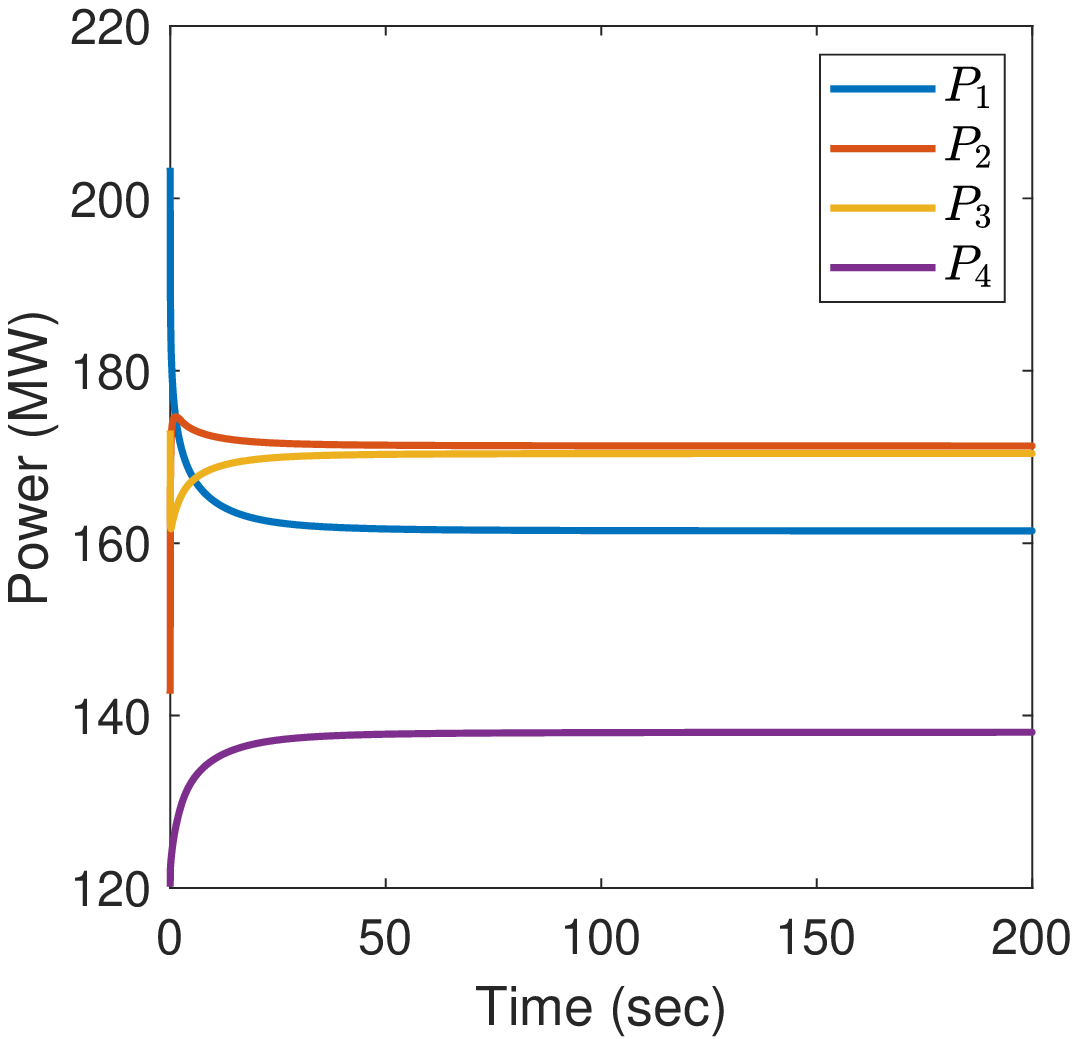}} \hspace*{-0.5cm}
	\subfigure[$P_L$ Vs time]{\includegraphics[width=0.34\textwidth]{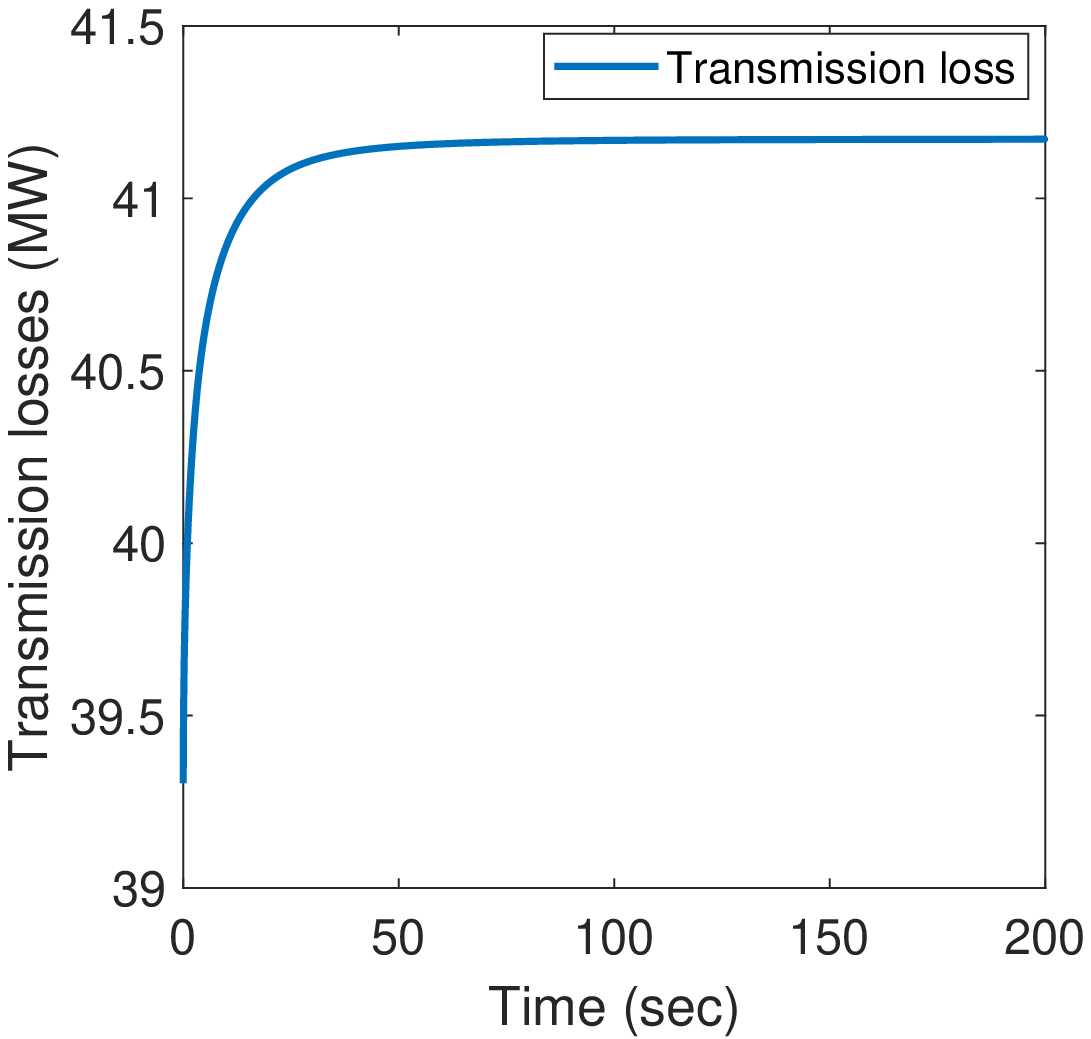}} \hspace*{-0.5cm}
	\subfigure[$\sum_{i = 1}^N P_i$ Vs time]{\includegraphics[width=0.35\textwidth]{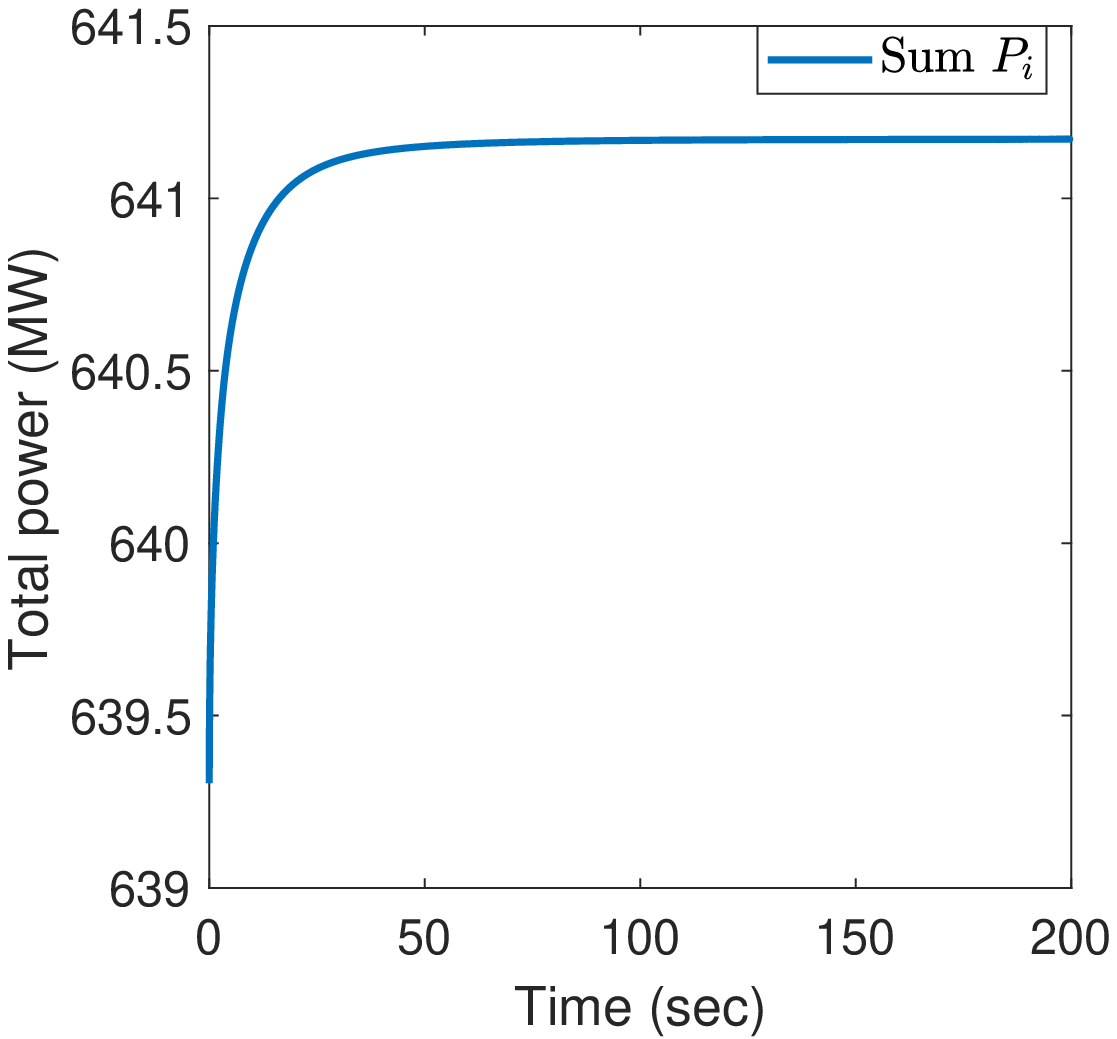}}
	\caption{Evolution of generated power, transmission losses and total power during $0-200$ sec. Clearly, the power supplied by generator network is equal to the sum of load demand and power transmission losses at each instant of time.} 
	\label{optimum_power}
\end{figure*}

\section{Simulation Example}\label{section4}
\begin{table}[]
\centering
\caption{Generator cost parameters.}
\begin{tabular}{|c|c|c|c|}
\hline
Bus   &  $a_i$($\$/$h)  &  $b_i$($\$/$MWh) &  $c_i$($\$/$MW$^2$h) \\ 
\hline\hline
 $G_1$  & 53      & 1.21         & 0.094          \\ \hline
 $G_2$ & 34       & 3.47         & 0.082       \\ \hline
 $G_3$ & 45       & 2.24         & 0.086      \\ \hline
 $G_4$ & 78       & 2.55         & 0.105       \\ \hline
\end{tabular}
\label{tab:caption}
\end{table}

Consider a power system network of four generators comprising a two-layered interaction topology, as shown in Fig.~\ref{network}. The generators share power outputs globally and the other auxiliary variables in algorithm \eqref{algorithm} are shared locally. The power generation cost associated with each generator is characterized by the quadratic function ${C}_i(P_i) = c_i P_i^2+ b_i P_i + a_i$, where $a_i, b_i, c_i$ are the cost coefficients. The economic dispatch problem can be described as: $\text{Min } {C}(P) = \sum_{i = 1}^4  c_i P_i^2 + b_iP_i + a_i$, {subject to } $\sum_{i = 1}^4 {P_i} = D+P_L$. The values of $a_i, b_i, c_i$ are given in Table \eqref{tab:caption}. Clearly, $\sigma = 2\min_i\{c_i\} =  0.164$ and $\delta = \min_i\{b_i\} = 1.21$. Let the total load demand be $D = 600$ MW. The initial power supplied by the generators are given by $P_1(0) = 170$ MW, $P_2(0) = 110$ MW, $P_3(0) = 140$ MW, $P_4(0) = 180$ MW. The power transmission losses \eqref{Transmission losses B loss} are obtained by setting the $\mathcal{B}-$loss coefficients as: 
\begin{equation*}
\mathcal{B}
=
\begin{bmatrix}
0.1200  &  0.0286  &  0.0481  &  0.0321\\
0.0286  &  0.1341  &  0.0511  &  0.1251\\
0.0481  &  0.0511  &  0.1539  &  0.1463\\
0.0321  &  0.1251  &  0.1463  &  0.1612
\end{bmatrix}
\times 10^{-3},
\end{equation*}
which is symmetric with all positive entries and $\mathcal{B}_0 = [2.0, 1.0, 2.5, 1.5]^T \times 10^{-3}; B_{00}=4$. It can be easily verified that the condition in Remark~\ref{B_coeff_justification} holds for the given load demand $D = 600$ MW, and the chosen $\mathcal{B}-$coefficients.  

\begin{itemize}
\item The algorithm \eqref{algorithm} is simulated with control parameters as $k_1 = k_2 = 5$, and $\mu = 0.5, \nu = 2$. The optimal power supplied by the generators are obtained as $P_1^* =161.4$ MW, $P_2^* = 171.3$ MW, $P_3^* = 170.4$ MW and $P_4^*=138.1$ MW, as shown in Fig.~\ref{optimum_power} (a). The total power supplied is $P_T = 641.2$ MW, meeting the load demand $D = 600$ MW and the power transmission losses $41.2$ MW at the optimal solution (see Figs.~\ref{optimum_power}(b) and \ref{optimum_power}(c)). One can observe from Fig.~\ref{optimum_power} that the demand and transmission losses are supplied by the generators at every instant of time. The optimal cost is plotted in Fig.~\ref{cost_function}, and is evaluated to be $\$ 11093$.

\item We verify the convergence time in these plots by evaluating the settling time $T_s$ in \eqref{equn 22}. For the given values of $\sigma, \delta$ and $\mathcal{B}-$coefficients, the matrix $\mathcal{S}$ in Lemma~\ref{Lemma3} is obtained as (considering each entry with four significant decimal places):
\begin{equation*}
\mathcal{S}
=
\begin{bmatrix}

    0.1646  &  0.0000 &   0.0001  &  0.0000 \\
    0.0000  &  0.1645  &  0.0001  &  0.0002 \\
    0.0001  &  0.0001  &  0.1648  &  0.0002 \\
    0.0000  &  0.0002  &  0.0002  &  0.1646 \\
\end{bmatrix}.
\end{equation*}
The minimum eigenvalues of $\mathcal{B}$ and $\mathcal{S}$ are $b_1 = -0.0161 \times 10^{-3}$ and $\tau_1 =  0.1644$, respectively. Further, $\rho= \min_i\{\mathcal{B}_{i0}\} = 1.0 \times 10^{-3}$ and the value of $\phi_2 = 0.5858$. It can be easily verified that $(1+\rho)\sigma + b_1\delta = (1+1.0 \times 10^{-3})\times0.164 + (-0.0161 \times 10^{-3})\times 1.21 = 0.1641 > 0$, satisfying the Assumption~(A2). Using the above values, the settling time is obtained as $T_s = 154.47$ sec, which supports our simulation results in Figs.~\ref{optimum_power} and \ref{cost_function}.

\item Furthermore, we have observed through simulations that the dynamics $\dot{z}_i$ is robust to the additive bounded disturbances with zero mean. That is, for any uniformly bounded zero-mean signal $w_i: \mathbb{R}_+ \to \mathbb{R}$ for each $i$, $\dot{z}_i =  - {k_1}\sig[\sum_{j \in {N_i}} {a_{ij}}({H_j}{\lambda _j} - {H_i}{\lambda _i})]^\mu - {k_2}\sig[{\sum_{j \in {N_i}} {{a_{ij}}({H_j}{\lambda _j} -{H_i}{\lambda _i})}}] ^\nu + w_i$ has no effect on the solution of algorithm~\eqref{algorithm}.   
\end{itemize}

\begin{figure}[t]
\centering 
\includegraphics[width=0.42\textwidth]{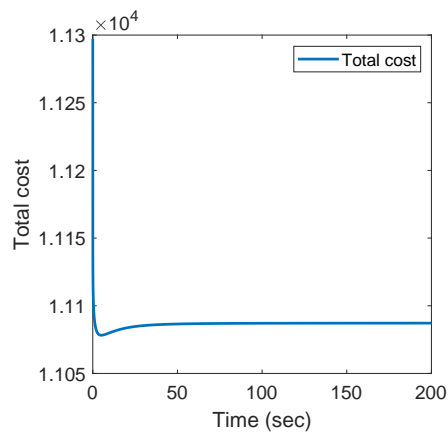} 
\caption{Evolution of $\sum_{i = 1}^N C_i (P_i)$ during $0–200$ s.}
\label{cost_function}
\end{figure}

%\begin{figure}[htbp]
%\centering 
%\includegraphics[width=0.6\textwidth]{ff2.eps} 
%{\caption{Evolution of the derivative of transmission losses $dP_L/dP_i$  during 0–30 s}} 
%\end{figure}

\section{Conclusion}\label{section5}
In this paper, we investigated the EDP with Kron's modeled power transmission losses, under a few assumptions on the $\mathcal{B}-$loss coefficients, network topology, and the convexity of the cost functions associated with each generator. The time-varying power transmission losses are incorporated in the equality constraints of considered EDP. It is shown that the proposed consensus-based (partially distributed) algorithm solves the EDP in a finite time, which is upper bounded by a term relying on the eigenvalues of the matrix $\mathcal{B}$, local Laplacian, and the constants describing the convexity of the cost functions. 

Although for the approximated power transmission losses, the proposed algorithm can be implemented in a fully distributed manner (see Remark~1). However, it remains a challenging problem to come up which such an algorithm accounting for Kron's modeled power transmission losses without an approximation. Besides, there are several possibilities for future work such as a) consideration of directed communication topology with time-delay in information sharing among generators b) incorporation of fluctuation in load demand.

%%%%%%%%%%%%%%%%%%%%%%%%%%%%%%%%%%%%%%%%%%%%%%%%%%%%%%%%%%%%%%%%%%%%%%%%%%%%%%%%%%%%%%%%%

\bibliographystyle{IEEEtran}
\bibliography{References}

%%%%%%%%%%%%%%%%%%%%%%%%%%%%%%%%%%%%%%%%%%%%%%%%%%%%%%%%%%%%%%%%%%%%%%%%%%%%%%%%%%%%%%%%%

\appendix

\subsection{Proof of Lemma~\ref{Lemma3}}

\begin{proof} 
The proof is provided sequentially for each step.
\begin{enumerate}[label={(R\arabic*)}]
\item By definition, $\mathcal{F} \in \mathbb{R}^N$ is a column vector (please notice the use of Hadamard product $\odot$) with $i^{\text{th}}$ entry $\mathcal{F}_{i} = ({\partial P_L}/{\partial P_i})({\partial {C}}/{\partial P_i})$, where, 
\begin{itemize}
	\item ${\partial P_L}/{\partial P_i}$ depends on both $P_i$ and $P_j$, according to Lemma~\ref{lemma3.1},
	\item ${\partial {C}}/{\partial P_i} = ({\partial {C}}/{\partial C_i})({\partial {C_i}}/{\partial P_i}) = {\partial {C_i}}/{\partial P_i}$ depends only on $P_i$ (as ${\partial {C}}/{\partial C_i} = 1$, according to \eqref{conditions 1}). 
\end{itemize}
Therefore, $\nabla \mathcal{F}$ is an $N \times N$ matrix, whose diagonal entries can be obtained using chain rule as
	\begin{align*}
	[\nabla \mathcal{F}_{ii}] &= \frac{\partial}{\partial P_i}\left[\frac{\partial P_L}{\partial P_i}\frac{\partial {C_i}}{\partial P_i}\right] = \frac{\partial}{\partial P_i}\left[\frac{\partial P_L}{\partial P_i}\right]\frac{\partial {C_i}}{\partial P_i} + \frac{\partial P_L}{\partial P_i}\frac{\partial^2 {C_i}}{\partial P^2_i},
	\end{align*}
which using Lemma~\ref{lemma3.1} can be written as
	\begin{align*}	
 	[\nabla \mathcal{F}_{ii}] &= 2\mathcal{B}_{ii}\left(\frac{\partial C_i(P_i)}{\partial P_i} + P_i\frac{\partial^2 C_i(P_i)}{\partial P^2_i}\right)\\
	& + \frac{\partial^2 C_i(P_i)}{\partial P^2_i}\left[\sum _{j = 1, j \ne i}^N {2{\mathcal{B}_{ij}}{P_j}} \right] +\mathcal{B}_{i0} \frac{\partial^2 C_i(P_i)}{\partial P^2_i}\\
	&= 2\mathcal{B}_{ii}\frac{\partial C_i(P_i)}{\partial P_i} + \frac{\partial^2 C_i(P_i)}{\partial P^2_i} \left(\mathcal{B}_{i0} + 2\mathcal{B}_{i0}P_i + \sum _{j = 1, j \ne i}^N {2{\mathcal{B}_{ij}}{P_j}} \right)\\
	& \geq 2\mathcal{B}_{ii}\frac{\partial C_i(P_i)}{\partial P_i}+ \mathcal{B}_{i0}\frac{\partial^2 C_i(P_i)}{\partial P^2_i}\\
	& \geq  2\mathcal{B}_{ii}\delta+\mathcal{B}_{i0}\sigma,
	\end{align*}
	under Assumptions~\ref{Assumption 2} and \ref{Assumption 3} for $P_i, P_j \geq 0, \forall i,j$. Similarly, the off-diagonal entries are given by
	\begin{align*}
	[\nabla \mathcal{F}_{ij}] & = \frac{\partial}{\partial P_j}\left[\frac{\partial P_L}{\partial P_i}\frac{\partial {C_i}}{\partial P_i}\right] = \frac{\partial {C_i}}{\partial P_i} \frac{\partial}{\partial P_j}\left[\frac{\partial P_L}{\partial P_i}\right] \\	
	& = 2\mathcal{B}_{ij}\frac{\partial C_i(P_i)}{\partial P_i}\geq 2\mathcal{B}_{ij}\delta. 
	\end{align*}
\item Clearly, $\mathcal{M} \in \mathbb{R}^N$ is a column vector with $i^{\text{th}}$ entry	$\mathcal{M}_{i} = [ \mathcal{B}_{ii} P_i + ({\mathcal{B}_{i0}}/{2})] ({\partial C_i(P_i)}/{\partial P_i})$, which depends only on $P_i$ for each $i$. As a result, $\nabla \mathcal{M}$ is an $N \times N$ diagonal matrix with diagonal entries
	\begin{align*}
	[\nabla \mathcal{M}_{ii}] &= \mathcal{B}_{ii}\frac{\partial C_i(P_i)}{\partial P_i}+\left[ \mathcal{B}_{ii}P_i+\frac{\mathcal{B}_{i0}}{2}\right] \frac{\partial^2 C_i(P_i)}{\partial P_i^2}\\
	&\geq \left[\mathcal{B}_{ii}\delta+\frac{\mathcal{B}_{i0}}{2}\sigma\right],
	\end{align*}
	for $P_i \geq 0, \forall i$ and following Assumptions~\ref{Assumption 2} and \ref{Assumption 3}.
\item The poof of this statement is straightforward and follows the similar steps as above. 
\item From Assumption~\ref{Assumption 3}, it is obvious that $\mathcal{S}$ is a symmetric matrix. Further, using (R3) it trivially holds that $\mathcal{S} \preccurlyeq \mathcal{Q}$.
\item Since $\mathcal{S}$ is a symmetric matrix, its eigenvalues are real and can be arranged as $\tau_1 \leq \tau_2 \cdots \leq \tau_N$. By construction, $\mathcal{S}$ can be written as the summation of two symmetric matrices $\sigma\text{diag}\{(1 + \mathcal{B}_{i0})\}$ and $\delta \mathcal{B}$, where the constants $\sigma$ and $\delta$ are defined in Assumption~\ref{Assumption 2}. Notice that the eigenvalues of $\sigma\text{diag}\{(1 + \mathcal{B}_{i0})\}$ are $[\sigma(1 + \mathcal{B}_{i0})]_{i=1}^{N}$, while for $\delta \mathcal{B}$ are $b_1 \delta \leq b_2 \delta \leq \cdots \leq b_N \delta$, if $\delta > 0$; and $b_N \delta \leq b_{N-1} \delta \leq \cdots \leq b_1 \delta$, if $\delta < 0$. Now, the results immediately follows by applying the Weyl's theorem [\cite{horn2012matrix}, Chapter~4, pg. 239].  
\end{enumerate}
\end{proof}

\end{document}